\documentclass[11pt]{amsart} 
\usepackage{amssymb}
\usepackage{amsmath}
\usepackage{latexsym}
\usepackage{amscd}
\usepackage{amsfonts}
\usepackage{pb-diagram}
\usepackage[ps,dvips,all]{xypic}
\usepackage[mathcal,mathscr]{eucal}
\usepackage{amsthm}
\usepackage{epsfig}
\usepackage{url}
\usepackage{pdfsync}
\usepackage{epsfig}
\usepackage{bbm}
\usepackage[pdftex,colorlinks]{hyperref}
\usepackage{pifont}
\usepackage{amsbsy}
\usepackage{epstopdf}
\usepackage{times}
\usepackage{epic,eepic
}       

\newtheorem{theorem}{Theorem}[section]
\newtheorem{prop}[theorem]{Proposition}
\newtheorem{lemma}[theorem]{Lemma}
\newtheorem{corollary}[theorem]{Corollary}

\theoremstyle{definition}
\newtheorem{definition}[theorem]{Definition}

\theoremstyle{remark}

\newtheorem{remark}[theorem]{Remark}

\numberwithin{equation}{section}

\DeclareMathOperator{\Hom}{Hom}

\DeclareMathOperator{\End}{End}
\DeclareMathOperator{\Ker}{Ker}

\DeclareMathOperator{\Ima}{Im}

\DeclareMathOperator{\dvol}{dvol}

\DeclareMathOperator{\supp}{supp}

\DeclareMathOperator{\tr}{tr}

\DeclareMathOperator{\id}{id}

\DeclareMathOperator{\Real}{Re}

\begin{document}
\title{Refined Kato inequalities  for harmonic fields on K\"{ahler} manifolds}

\author{Daniel Cibotaru}
\address{Instituto de Matem\'{a}tica e Estat\'{i}stica, Universidade Federal Fluminense, Niteroi, RJ, Brasil}
\email{daniel@mat.uff.br}

\author{Peng Zhu}
\address{School of Mathematical Science, Yangzhou University, Yangzhou, Jiangsu 225002, P. R. China}
\email{zhupeng2004@yahoo.com.cn}

\thanks{The authors are supported by the CAPES(PNPD) program.\\
    $~~~~~~~~~~~$The second author was partially supported by NSF
Grant (China) 11071208, Tian Yuan Fund of Mathematics
(China)11026116 and Yangzhou University fund 2010CXJ001}

\begin{abstract}{We obtain a refined Kato inequality for closed and coclosed differential $(p,q)$ forms  on  a  K\"{a}hler manifold.}
\end{abstract}

\maketitle

\section{Introduction}
 Kato inequalities have been shown to be important technical tools which are used to prove analytico-geometric results. In the articles \cite {Br,CaGaHe}, Branson, Calderbank, Gauduchon and Herzlich study injectively elliptic Stein-Weiss operators and they show that the sections in the kernel of such an operator satisfy improved Kato inequalities with constants that can be determined from representation theoretic data. 
 
 In \cite[Theorem 6.3(ii)]{CaGaHe} (the case $k=1$) the authors prove such an inequality for differential forms in the kernel of the Hodge-de Rham operator $d+d^*$ on a Riemannian manifold. This result is also stated in \cite{Wa} as Lemma 4.2, which contains an omission, namely the condition that the degree has to be at most half the dimension of the manifold (which is what the author actually needs). 
 
 The purpose of this article is to further refine this  Kato inequality for  forms of type  $(p,q)$  on a K\"{a}hler  manifold. We did this in Theorem \ref{Kahler} for all values of $p$ and $q$ except for $p=q$. The most important consequence of Theorem \ref{Kahler} is, in our view,  a Kato inequality for  holomorphic forms  on all K\"{a}hler manifolds (see Corollary \ref{holforms}).  
 
  Our presentation follows closely the methods of Branson, Calderbank, Gauduchon and Herzlich. In fact, we will present a proof of  the mentioned result from  \cite{CaGaHe} avoiding as much as possible the representation theoretic technicalities. It is this proof that suggested the improvement in the K\"{a}hler case.
  
  In the case of  complete K\"{a}hler manifolds, Kong, Li and Zhou in \cite{KoLiZh}   and Lam in \cite{La} showed that an $L^2$ harmonic
1-form $\omega$ has to satisfy
$$|d|\omega||\leq\frac{1}{\sqrt{2}}|\nabla\omega|.$$  In Corollary \ref{refineK}  we reprove this, with  the methods introduced here. In \cite{Wa2}, Wang proves an inequality about {\it real}, closed and coclosed $(1,1)$ forms   with a constant sharper  than in the Riemannian case.

We would like to thank Prof. Detang Zhou for suggesting the problem to us and for useful conversations. In an earlier version, we wrongly claimed that the proof of the refined Kato inequality in the Riemannian case has not appeared before. We would like to  thank  X. Wang for pointing it out  to us.

\section{Stein-Weiss operators}\label{SteinWeiss}
Let $M$ be a Riemannian manifold of dimension $n$, not necessarily compact.  We will call {\it harmonic fields} the forms $\omega\in C^{\infty}(M;\Lambda^k T^*M)$ that satisfy $(d+d^*)\omega=0$. If $M$ is compact then the harmonic fields coincide with the harmonic forms, i.e. solutions of $\Delta \omega=0$.

We introduce the main type of operators. 
Let $E$ be a  vector space endowed with an inner product. We suppose that $E$ is a real representation of $SO(n)$ given by $\rho:SO(n)\rightarrow GL(E)$ and let $F\hookrightarrow \mathbb{R}^n\otimes E$ be subrepresentation of the  canonical representation tensored with $E$ and we denote by $\Pi$ the orthogonal projection
\[ \Pi: \mathbb{R}^n\otimes E\rightarrow F.
\] 
We will use the same letters to denote the projection of vector bundles over $M$,
\[ \Pi:T^*M\otimes E\rightarrow F.
\]
The Levi-Civita connection on the frame bundle $P_{SO}M$ of $M$ induces a connection $\nabla$ on $E$. 
\begin{definition} A Stein-Weiss (gradient) operator is a first order differential operator $L:\Gamma(E)\rightarrow \Gamma(F)$ obtained as the composition
\[  \Gamma(E)\rightarrow \Gamma(T^*M\otimes E)\rightarrow \Gamma(F), \quad\quad L:= \Pi\circ \nabla.
\]
\end{definition}
\begin{remark} A beautiful study of Stein-Weiss operators can be found in \cite{Br2} where T. Branson classifies those operators which are injectively elliptic (see Definition \ref{injell}). 
\end{remark}
Since a Stein-Weiss operator is essentially built from two objects: an orthogonal projection morphism   of $SO(n)$ representations and a connection  on the manifold, by a slight abuse of terminology and notation we will talk about the composition of  these  two objects  instead of the more lengthy expression "the composition of the connection with the associated projection of vector bundles".

\begin{definition} The rescaled Hodge-de Rham operator is the operator which on $k$-forms acts as
\[ 
\frac{1}{\sqrt{k+1}}d+\frac{1}{\sqrt{n+1-k}}d^*.
\]

\end{definition}
Notice that the harmonic fields can be seen as the solutions of the rescaled Hodge-de Rham equation.

\begin{prop}\label{rescSW} The rescaled Hodge-de Rham operator  is a Stein-Weiss operator. \end{prop}
\begin{proof} We will show that 
\[ \Pi: \mathbb{R}^n\otimes\Lambda^k\mathbb{R}^n\rightarrow\Lambda^{k+1}\mathbb{R}^n\oplus\Lambda^{k-1}\mathbb{R}^n, \quad\quad \Pi(\eta\otimes\omega)=\left(\frac{1}{\sqrt{k+1}}\eta\wedge \omega, -\frac{1}{\sqrt{n-k+1}} \iota_{\eta^*}\omega\right)
\]
is a morphism of $SO(n)$ representations, where $\iota_{\eta^*}$ represents contraction with the metric-dual to $\eta$. In fact, $\Pi$ is the orthogonal projection of the tensor product representation  $\mathbb{R}^n\otimes\Lambda^k\mathbb{R}^n$ to a direct sum of two subrepresentations. 

Let $\{e_i, ~i=1\ldots n\}$ be  an orthogonal basis of $\mathbb{R}^n$. Let
 \begin{equation}\label{thetalabel1} \theta_{1}: \Lambda^{k+1}\mathbb{R}^n\rightarrow\mathbb{R}^n\otimes\Lambda^k\mathbb{R}^n, \end{equation}
 \[\theta_1(v_1\wedge \ldots \wedge v_{k+1})=\frac{1}{\sqrt{k+1}}\sum_{i=1}^{k+1}(-1)^{i-1}v_i\otimes v_1\wedge\ldots \hat v_i \wedge\ldots v_{k+1}\]
and let 
\begin{equation}\label{thetalabel2} \theta_{2}: \Lambda^{k-1}\mathbb{R}^n\rightarrow\mathbb{R}^n\otimes\Lambda^k\mathbb{R}^n,\quad\quad \theta_2(\omega)=-\frac{1}{\sqrt{n-k+1}}\sum_{i=1}^{n} e_i\otimes (e_i\wedge \omega).
\end{equation}
be two maps. It is easy to check that they are injective morphisms of $SO(n)$-representations.  The first one is obviously so, while the second can be described as the composition of  
\[ \mathbb{R}^n\otimes\mathbb{R}^n\otimes\Lambda^{k-1}\mathbb{R}^n\rightarrow\mathbb{R}^n\otimes\Lambda^{k}\mathbb{R}^n,\quad\quad\quad \xi\otimes\eta\otimes\omega\mapsto \xi\otimes\eta\wedge\omega
\] with
\[\Lambda^{k-1}\mathbb{R}^n\rightarrow \mathbb{R}^n\otimes\mathbb{R}^n
\otimes\Lambda^{k-1}\mathbb{R}^n,\quad \quad\quad \omega\mapsto -\frac{1}{\sqrt{n-k+1}}\left(\sum_{i=1}^{n}e_i\otimes e_i\right)\otimes\omega.\]
 Note that if $A\in SO(n)$ 
\[ \sum_{i=1}^{n} A e_i\otimes Ae_i =\tr(A^TA)\sum_{i=1}^{n} e_i\otimes e_i=\sum_{i=1}^{n} e_i\otimes e_i.
\]
The presence of  the constants $1/\sqrt{k+1}$ and $1/\sqrt{n-k+1}$ is a reminder of the fact that we are dealing with {\it isometric} monomorphisms of representations. The following relations hold
\[ \Pi_1\circ \theta_{1}=\id_{\Lambda^{k+1}\mathbb{R}^n},\quad\quad\quad\quad \Pi_2\circ\theta_{2}=\id_{\Lambda^{k-1}\mathbb{R}^n}.
\]
Moreover  we have that
\[ \Ker{\Pi_1}= \theta_{1}( \Lambda^{k+1}\mathbb{R}^n)^\perp\quad\quad\mbox{and}\quad\quad \Ker{\Pi_2}= \theta_{2}( \Lambda^{k-1}\mathbb{R}^n)^\perp. 
\]
Indeed, using the identities from Lemma \ref{linalg} (below) one can prove the $\subset$ inclusions which is enough because $\Pi_1$ and $\Pi_2$ are surjective.

Another easy application of Lemma \ref{linalg} shows that the images of $\theta_1$ and $\theta_2$ are orthogonal from which we deduce that $\Pi$ is the orthogonal projection onto the $SO(n)$-invariant subspace \linebreak  $\Lambda^{k+1}\mathbb{R}^n\oplus\Lambda^{k-1}\mathbb{R}^{n-1}\hookrightarrow\mathbb{R}^n\otimes\Lambda^k\mathbb{R}^n$.

 We combine what we have just said with the well known result (see Proposition 1.22 and Proposition 2.8 in \cite{BGV}) that the Hodge-de Rham operator  is the composition of the Clifford multiplication
\[ c:T^*M\otimes \Lambda^*T^*M\rightarrow \Lambda^*T^*M,\quad\quad c(\eta\otimes\omega)=(\eta\wedge\omega,-i_{\eta^*}(\omega))
\]
with the Levi-Civita connection  and we are done.
\end{proof}
The following lemma was used in the proof above.
\begin{lemma}\label{linalg}  The following identities hold:
\begin{enumerate}
\item[(a)] $\langle\xi\wedge\omega,\eta_1\wedge\ldots\wedge\eta_{k+1} \rangle=\displaystyle\sum_{i=1}^{k+1}(-1)^{i-1}\langle \xi,\eta_i\rangle\cdot\langle\omega,\eta_1\wedge\ldots \hat{\eta}_i\ldots\wedge\eta_{k+1}\rangle$, \\ where $\xi,\eta_i\in\mathbb{R}^n,~\omega\in\Lambda^k\mathbb{R}^n;$
\item[(b)] $\langle\iota_{\xi}(\omega),\theta\rangle=\displaystyle\sum_{i=1}^n\langle\xi,e_i\rangle\cdot\langle\omega, e_i\wedge\theta\rangle$, where $\xi\in\mathbb{R}^n$, $\theta\in\Lambda^{k-1}\mathbb{R}^n$ and $\omega\in\Lambda^{k}\mathbb{R}^n$.
\end{enumerate}
\end{lemma}
\begin{proof} Let $I\subset\{1,\ldots, n\}$ be a subset with
\begin{itemize}
\item [(a)] $k+1$ or 
\item[(b)] $k-1$
\end{itemize} elements and let $e_I:=\wedge_{i\in I}e_i$. By linearity, it is enough to prove the identities for $\eta_1\wedge\ldots\wedge\eta_{k+1}=e_I$ and $\theta=e_I$, respectively. Also it is enough to consider $\omega=\alpha_1\wedge\ldots\wedge \alpha_k$.

In this situation, the number on the left hand side of the first identity is the $(k+1)\times (k+1)$ minor formed by taking the $I$-columns in the $n\times (k+1)$ matrix having the entries of $\xi$ on the first row and $\alpha_1,\ldots ,\alpha_k$ on the  next ones. The identity itself is stating the well known fact that this minor can be computed  as an  alternating sum of the relevant entries of $\xi$ multiplied with the corresponding $k\times k$ minors with entries from the matrix made of $\alpha_1,\ldots,\alpha_k$.

For the second identity, we further simplify by letting $\xi:=e_p$. Then we have to prove that
\[ \sum_{i=1}^{k} (-1)^{i-1}\langle e_p,\alpha_i \rangle\cdot\langle\alpha_1\wedge\ldots \hat{\alpha}_i\ldots\wedge \alpha_{k},e_I\rangle=\langle\alpha_1\wedge\ldots\wedge\alpha_k,e_p\wedge e_I\rangle, 
\]
which is nothing but the first identity for $\xi=e_p$ and $\omega=e_I$.
\end{proof}

A natural question is whether  the above rescaling of the Hodge-de Rham operator is the only one that turns it into a  Stein-Weiss operator. This is true up to some signs. 

\begin{lemma}\label{linalg2} Let $\sigma:V\rightarrow W$ be a linear map between inner product spaces (real or complex). Let  $\alpha:W\hookrightarrow V$  be an isometry such that $a\sigma\circ \alpha=\id_{W}$ for some constant $a$ and $\sigma(v)=0$ for all $v\in\alpha(W)^\perp$.  Suppose there exists another constant $b$ and an  isometry $\beta:W\hookrightarrow V$ such that $b\sigma\circ\beta=\id_{W}$ and $\sigma(v)=0$ for all $v\in\beta(W)^\perp$. Then $b=\mu a$ where $|\mu|=1$.
\end{lemma}
\begin{proof} The conditions in the lemma express the fact that $a\sigma$ and $b\sigma$ are orthogonal projections onto $\alpha(W)$ and $\beta(W)$, respectively. Let $A,B: V\rightarrow V$ be the orthogonal projections seen as endomorphisms of $V$.
\[A:=a\alpha\circ\sigma, \quad\quad\quad\quad B:=b\beta\circ\sigma.\]
Now, there exists an orthogonal transformation $T:V\rightarrow V$ such that $\beta=T\circ\alpha$, hence
\[ B=\frac{b}{a}T\circ A
\] The relation $B=B^*$ implies that $\frac{b}{a}A\circ T=\frac{\bar{b}}{\bar{a}}T^*\circ A$ which fed into 
\[ B^2=\left(\frac{b}{a}\right)^2T\circ (A\circ T)\circ A=B
\] 
gives $\frac{|b|^2}{|a|^2} A^2=\frac{|b|^2}{|a|^2}A=B$. Hence the image of $B$ is the same as the image of $A$ and since they are orthogonal projections we must have $A=B$ and so $|b|^2=|a|^2$.
\end{proof}
\begin{corollary} The operators $\pm 1/\sqrt{k+1}d$ and $\pm1/\sqrt{n-k+1}d^*$ are the only multiples of $d$ and $d^*$ which are Stein-Weiss.
\end{corollary}

\section{kato constants}\label{Kato const}
The classical Kato inequality states that if $\phi$ is a section of the vector bundle $E$ then the following inequality holds away from the set  $\phi^{-1}(0)$
\[  |d|\phi||\leq |\nabla \phi|.
\]
The equality above takes place when  $\nabla \phi=\xi\otimes \phi$ for some $1$-form $\xi$.
\begin{definition} A refined Kato inequality is an inequality of the type
\[  |d|\phi|| \leq\alpha |\nabla \phi|
\] with $\alpha<1$, which will be called a refined Kato constant.
\end{definition}
The main insight of \cite{Br} and \cite{CaGaHe} is that if $\phi$ is a section in the kernel of an injectively elliptic Stein-Weiss operator $L$ then $\phi$ satisfies a  Kato inequality which  is stronger than the classical one. Moreover the refined Kato constant $\alpha$ depends only on the symbol of $L$.

We give now the relevant definitions. We will work in a slightly more general context than in the previous section, i.e. the bundle  $E$ will be complex and $L$ will be a complex differential operator. The Stein-Weiss operators we considered above are real operators. However by complexifying the representations one can obtain complex operators. They should really be called $SO(n)$ (real or complex) Stein-Weiss operators since they are associated to the $SO(n)$-frame bundle on the manifold. In the next section we will consider $U(n)$ Stein-Weiss operators when the manifold $M$ is K\"{a}hler. 
\begin{definition} The {\it symbol} of a complex differential operator $L:\Gamma(E)\rightarrow\Gamma(F)$ of order $k$ is the map
 \[  \sigma(L):T^*M\rightarrow \Hom(E,F)\quad\quad \sigma(L)(\xi\otimes\ldots \otimes\xi):=i^{k}\frac{1}{k!}[\ldots[L,f],\ldots, f],\quad\quad \xi:=df .\] The symbol of a real differential operator is the symbol of its complexification.
\end{definition}
The symbol of the Stein-Weiss operator $L:\Gamma(E)\rightarrow\Gamma(F)$ is the morphism of bundles
\[ i\Pi_F:T^*M\rightarrow \Hom(E, F).
\]
The symbol of the formal adjoint $L^*:\Gamma(F)\rightarrow\Gamma(E)$ is 
\[ \sigma(L^*):T^*M\rightarrow \Hom(F,E),\quad \quad \sigma(L^*)=\sigma(L)^*=-i\Pi_{F}^*\] and the symbol of $L^*L$ is $\Pi_F^*\Pi_F$.
\begin{definition}\label{injell} An operator $L$ is called \textit{injectively elliptic} if $L^*L$ is elliptic.  

Let $L$ be an operator of order $1$. A number $\epsilon$ is called a constant of ellipticity for $L^*L$ if the following relation holds
\[ \langle \sigma_{\xi}(L^*L)(v),v\rangle \geq \epsilon |\xi|^2|v|^2,\quad\quad \forall \xi\in T^*M, ~v\in E. 
\]
\end{definition}
If $L^*L$ is elliptic then the positivity of $\sigma(L^*L)$ implies that an ellipticity constant exists, at least locally. In the case of an injectively elliptic Stein-Weiss operator this constant is smaller than $1$ because of the next straightforward lemma and the fact that the Bochner Laplacian $\nabla^*\nabla$ has constant of ellipticity equal to $1$. 
\begin{lemma} Let $F^\perp$ be the orthogonal complement of $F$ in $\mathbb{R}^n\otimes E$ and let $L_{\perp}$ be the corresponding Stein-Weiss operator. Then 
\[ L^*L+L_{\perp}^*L_{\perp}=\nabla^*\nabla.\] 
\end{lemma}
\begin{proof} One uses  the equality $\nabla=L\oplus L_{\perp}$.
\end{proof}

The connection between constants of ellipticity and refined Kato constants is provided by the following:
\begin{lemma} \label{elliptKato} Let $\mathbb{C}^n\otimes E=F\oplus F^{\perp}$ and let $\Pi$ and $\Pi^{\perp}$ be the corresponding orthogonal projections  onto $F$ and $F^\perp$ respectively. If $\phi\in \Ker(\Pi\circ \nabla)$ then
\[   |d |\phi| |\cdot| \phi|   \leq     | \nabla \phi | \cdot |\Pi^{\perp} (\xi_0\otimes\phi) |
\] 
for some one form $\xi_0$ which is real and of norm 1.
\end{lemma}
\begin{proof}  We have
\[ d|\phi|^2=2\Real\langle \nabla\phi,  \phi \rangle.
\]
Let $\xi_0:=d|\phi|^2/|d|\phi|^2|$. Then
\[ 2 |d |\phi| |\cdot| \phi|= |d|\phi|^2|=\langle d|\phi|^2, \xi_0 \rangle=2\langle \Real\langle \nabla \phi,\phi \rangle  , \xi_0\rangle\stackrel{*}{=}2\Real\langle \langle \nabla \phi,\phi \rangle,\xi_0\rangle=\]
\[ =2\Real\langle \nabla \phi ,\xi_0\otimes \phi  \rangle=2\Real\langle \Pi^{\perp}\circ\nabla (\phi) , \xi_0\otimes \phi \rangle=2\Real\langle \nabla \phi ,\Pi^{\perp}(\xi_0\otimes \phi ) \rangle\leq 2 | \nabla \phi | \cdot |\Pi^{\perp} (\xi_0\otimes\phi ) |,
\]
where $*$ equality holds because $\xi_0$ is real.
\end{proof}

Let us notice that 
\[ |\phi|^2=|\xi_0\otimes\phi|^2=|\Pi(\xi_0\otimes\phi)|^2+|\Pi^{\perp}(\xi_0\otimes\phi)|^2=\langle \Pi^*\Pi(\xi_0\otimes\phi), \xi_0\otimes\phi\rangle+|\Pi^{\perp}(\xi_0\otimes\phi)|^2.
\]
where $\Pi^*\Pi$ is the symbol of $L^*L$. So if $L^*L$ is injectively elliptic with a constant of ellipticity $\epsilon$ then 
\[|\Pi^{\perp}(\xi_0\otimes\phi)|^2\leq (1-\epsilon) |\phi|^2.
\]
In combination with  Lemma \ref{elliptKato} we have the following
\begin{prop}\label{constell} Let $L$ be an injectively elliptic Stein-Weiss operator. If $\epsilon$ is a constant of ellipticity for $L^*L$, then $\alpha=\sqrt{1-\epsilon}$ is a refined Kato constant for $\phi\in \Ker{L}$.
\end{prop}
\begin{lemma}\label{constellhod}  A constant of ellipticity for the rescaled Hodge-de Rham operator acting on $k$-forms is 
\[ \epsilon=\left\{ \begin{array}{ccc} 1  &\mbox{if} &k=0,n;\\
\min{\left\{\frac{1}{k+1},\frac{1}{n-k+1}\right\}} &\mbox{if}  &1\leq k\leq n-1. \end{array}\right. \]
\end{lemma}
\begin{proof} Let  $1\leq k\leq n-1$. If $L=\frac{1}{\sqrt{k+1}}d+\frac{1}{\sqrt{n+1-k}}d^*$ then 
\[ L^*L=\frac{1}{k+1}d^*d+\frac{1}{n+1-k}dd^*\geq \min{\left\{\frac{1}{k+1},\frac{1}{n-k+1}\right\}}\Delta
\]
and the Laplacian has constant of ellipticity $1$.  For $k=0$ and $k=n$ the operator $L^*L$ is just the Laplacian on functions and on top degree forms, respectively.

A second simple proof can be provided using  the symbol of $L^*L$ and the Cartan formula
\[ e_u\iota_u+\iota_ue_u=|u|^2\id,
\] where $e_u$ and $i_u$ are exterior multiplication and contraction by $u$, respectively.
\end{proof}
Putting together Proposition \ref{constell} and Lemma \ref{constellhod} we have  \cite[Theorem 6.3(ii)]{CaGaHe}  $(k=1)$.
\begin{theorem}[Calderbank-Gauduchon-Herzlich]\label{main} Let $\omega$ be a $k$-form in the kernel of $d+d^*$. Then $\omega$ satisfies the refined Kato inequality
\[ |d|\omega|| \leq \sqrt{\frac{n-k}{n-k+1}}|\nabla \omega| \quad\mbox{ if } \quad1\leq k\leq n/2,
\] and
\[|d|\omega||\leq \sqrt{\frac{k}{k+1}}|\nabla \omega| \quad\mbox{ if } \quad  n/2\leq k\leq n-1,
\]
while for $k=0,n$ the form $\omega$ is parallel. 
\end{theorem}
\begin{remark} The question of sharpness in the previous inequalities depends in general on the manifold under consideration.  For example, if the manifold $M$ is compact and symmetric then every harmonic field is parallel, hence the best constant  in this case is $0$. 

 In general, to have equality above one first needs equality in Lemma \ref{elliptKato} which after a quick inspection implies that the form $\omega$ has to satisfy the relation 
\[ \nabla \omega=\xi\otimes\omega -\theta_1\left(\frac{1}{\sqrt{k+1}}\xi\wedge\omega\right)-\theta_2\left(-\frac{1}{\sqrt{n-k+1}}\iota_{\xi}(\omega)\right)
\]
for some $1$-form $\xi$, where $\theta_1$ and $\theta_2$ were defined at (\ref{thetalabel1}) and (\ref{thetalabel2}). More importantly, $\omega$ has to be a harmonic field. On the other hand, Branson shows in Theorem $7$ of \cite{Br} that such a form exists on flat $\mathbb{R}^n$.  In his proof, it is essential that $E$ is an irreducible representation of $SO(n)$. Branson's example is not $L^2$ integrable, hence it is conceivable that the inequalities above can be further refined if one imposes such a global condition. 
\end{remark}

\begin{remark} Notice that the inequalities in the previous theorem respect Poincar\'{e} duality in the sense that the same refined Kato constant works both for $k$ and for $(n-k)$ -forms. One has to  expect this because of the next basic result.
\end{remark}
\begin{lemma}\label{star Hodge} The star Hodge operator is an isometry and a parallel endomorphism of $\Lambda^*T^*M$. 
\end{lemma}
\begin{proof} The fact that it is an isometry is standard. To prove that it is parallel one first easily shows that the volume form  $\dvol$ is parallel. This follows by differentiating
\[ |\dvol|^2=1.
\]
which implies that $\langle \nabla_X\dvol,\dvol\rangle=0$ for all vector fields $X$. Then we apply $\nabla_X$ to the following pointwise equality which defines the Hodge star operator
\[\eta\wedge *\omega=\langle \eta,\omega\rangle\dvol\quad\quad \forall\eta,\omega\in\Gamma(\Lambda^kT^*M)\]
 to get 
\[\nabla_X\eta\wedge*\omega+\eta\wedge\nabla_X(*\omega)=(X\langle\eta,\omega\rangle)\dvol+\langle \eta,\omega\rangle\nabla_X{\dvol}.
\]  Hence $\eta\wedge(\nabla_X*\omega)=\langle\eta,\nabla_X\omega\rangle\dvol=\eta\wedge *\nabla_X\omega$, for all $\eta,\omega\in \Gamma(\Lambda^kT^*M)$.
\end{proof}
\begin{remark} The theory does not provide an inequality for harmonic forms, only for forms $\omega$ in the kernel of $d+d^*$, the so called harmonic fields. However in conjunction with an $L^2$ bound on $\omega$ one knows (see Proposition \ref{Yau}) that the harmonic fields are the same as the harmonic forms which is the case when $M$ is complete, for example.
\end{remark}

\section{The K\"{a}hler case}
In the K\"{a}hler case, using essentially the same theory, we get a better Kato constant for harmonic fields that respects Hodge duality. In what follows, $M$ is a K\"{a}hler manifold of complex dimension $n$. Notice that we can talk about the unitary frame bundle of $M$ and about $U(n)$ Stein-Weiss operators which are defined exactly as in Section \ref{SteinWeiss} by replacing $SO(n)$-representations with $U(n)$-representations and morphisms thereof.

 The result is as follows.
\begin{theorem}\label{Kahler} Let $0\leq p,q\leq n$  and let $\omega\in\Gamma(\Lambda^{p,q}T^*M)$ such that $(d+d^*)\omega=0$. Let $\alpha\geq0$ be such that
\[ \alpha^2:=\left\{\begin{array}{cc} 
\frac{1}{2} \quad&\mbox{if } p\in\{0,n\}\mbox{ or } q\in\{0,n\};\\
 \min{\left\{\max{\left\{\frac{2p+1}{2p+2},\frac{2n-2p+1}{2n-2p+2}\right\}},\max{\left\{\frac{2q+1}{2q+2},\frac{2n-2q+1}{2n-2q+2}\right\}}\right\}}& \mbox{ otherwise}.\end{array}\right.
\] Then
\[|d|\omega||\leq \alpha|\nabla \omega|.
\]  
\end{theorem}
\begin{proof} We decompose 
\[ d=\partial+\bar{\partial} \quad\quad \mbox{ and }\quad\quad d^*=\partial^*+\bar{\partial}^*.
\]
Notice that since $\omega\in\Gamma(\Lambda^{p,q}T^*M)$ we have $(d+d^*)\omega=0$ is equivalent with 
\[\partial\omega=\partial^*\omega=\bar{\partial}\omega=\bar{\partial}^*\omega=0.\]
We will see that
\begin{equation} \label{UnSW} L_1:=\frac{1}{\sqrt{p+1}}\partial+\frac{1}{\sqrt{n-p+1}}\partial^*\quad\quad \mbox{ and }\quad\quad L_2:=\frac{1}{\sqrt{q+1}}\bar\partial+\frac{1}{\sqrt{n-q+1}}\bar\partial^*
\end{equation}
are $U(n)$ Stein-Weiss operators.
Since the manifold is K\"{a}hler, each of the vector bundles $\Lambda^{p,q}$ comes endowed with a Levi-Civita connection. One can write down the symbols of each of the operators $\partial,\bar{\partial},\partial^*,\bar{\partial}^*$ (compare with Proposition 3.67 in \cite{BGV}).
\[ \sigma(\partial):T^*M\rightarrow \End(\Lambda^{p,q}T^*M,\Lambda^{p+1,q}T^*M),\quad\quad \sigma(\partial)_{\xi}(\omega)=i\xi^{1,0}\wedge\omega;\]
\[   \sigma(\bar{\partial}):T^*M\rightarrow \End(\Lambda^{p,q}T^*M,\Lambda^{p,q+1}T^*M),\quad\quad  \sigma(\bar{\partial})_{\xi}(\omega)=i\xi^{0,1}\wedge\omega;
\]
\[ \sigma(\partial^*):T^*M\rightarrow \End(\Lambda^{p,q}T^*M,\Lambda^{p-1,q}T^*M),\quad\quad \sigma(\partial^*)_{\xi}(\omega)=(-i)\iota_{(\xi^{0,1})^*}(\omega);\]
\[ \sigma (\bar{\partial}^*):T^*M\rightarrow \End(\Lambda^{p,q}T^*M,\Lambda^{p,q-1}T^*M),\quad\quad \sigma (\bar{\partial}^*)_{\xi}(\omega)=(-i)\iota_{(\xi^{1,0})^*}(\omega),
\]
where $\iota$ represents contraction and $((\xi^{0,1})^*,(\xi^{1,0})^*)\in T^{1,0}M\oplus T^{0,1}M$ is the metric dual of $\xi=(\xi^{1,0},\xi^{0,1})\in T^*_{\mathbb{C}}M$. Notice that for $\xi\in T^*M$ we have $\xi^{0,1}=\overline{\xi^{1,0}}$.

We want to show that $(-i)1/\sqrt{p+1}\sigma(\partial)$ when seen as a linear map defined on $T^*_{\mathbb{C}}M$ is an associated bundle morphism to an orthogonal projection of $U(n)$ representations. Analogous statements hold for the other three maps. 

 Let $\bar{\mathbb{C}}^{n}=(\mathbb{R}^{2n},-i)$ be the conjugate of the standard complex space. The standard action of $U(n)$ is complex linear on $\bar{\mathbb{C}}^{n}$ and the standard Hermitian metric on $\mathbb{C}^n$ builds an isomorphism of $U(n)$ representations between $\bar{\mathbb{C}}^n$ and $(\mathbb{C}^n)^*$.  The bundles $\Lambda^{p,q}T^*M:=\Lambda^{p}T^{1,0}M^*\otimes \Lambda^{q}T^{0,1}M^*$ are associated bundles to the unitary frame bundle of $M$ (induced by the Riemannian metric and the complex structure) and the canonical representations of $U(n)$ on $\Lambda^{p}\bar{\mathbb{C}}^n\otimes\Lambda^q\mathbb{C}^n$. 

 Let $e_i\in \mathbb{C}^n$, $i=1\ldots n$ be an orthonormal basis with respect to the standard Hermitian metric and denote by $\bar{e}_i\in\bar{\mathbb{C}}^n$, $i=1\ldots n$ the dual, or conjugate basis. 
We define 
\[  \theta^{\partial}: \Lambda^{p+1}\bar{\mathbb{C}}^n\otimes \Lambda^q\mathbb{C}^n\hookrightarrow (\bar{\mathbb{C}}^n\oplus \mathbb{C}^n)\otimes\Lambda^{p}\bar{\mathbb{C}}^n\otimes\Lambda^{q}\mathbb{C}^n,\]\[ \theta^{\partial}(\omega_1\wedge\ldots \wedge \omega_{p+1}\otimes\eta)=\frac{1}{\sqrt{p+1}}\sum_{i=1}^{p+1}(-1)^{i-1}(\omega_i,0)\otimes \omega_1\wedge\ldots \hat{\omega}_i\ldots\wedge \omega_{p+1}\otimes\eta;
\]
\[ \theta^{\bar{\partial}}  :\Lambda^{p}\bar{\mathbb{C}}^n\otimes \Lambda^{q+1}\mathbb{C}^n\hookrightarrow (\bar{\mathbb{C}}^n\oplus \mathbb{C}^n)\otimes\Lambda^{p}\bar{\mathbb{C}}^n\otimes\Lambda^{q}\mathbb{C}^n,\]\[ \theta^{\bar{\partial}}(\omega\otimes\eta_1\wedge\ldots\wedge \eta_{q+1})=\frac{1}{\sqrt{q+1}}\sum_{i=1}^{q+1}(-1)^{i-1}(0,\eta_i)\otimes \omega\otimes\eta_1\wedge\ldots \hat{\eta}_i\ldots\wedge \eta_{q+1};
\]
\[\theta^{\partial^*}: \Lambda^{p-1}\bar{\mathbb{C}}^n\otimes \Lambda^q\mathbb{C}^n\hookrightarrow (\bar{\mathbb{C}}^n\oplus \mathbb{C}^n)\otimes\Lambda^{p}\bar{\mathbb{C}}^n\otimes\Lambda^{q}\mathbb{C}^n,
\]
\[\theta^{\partial^*}(\omega\otimes\eta)=-\frac{1}{\sqrt{n-p+1}}\sum_{i=1}^{n}(0,e_i)\otimes\bar{e}_i\wedge\omega\otimes\eta;
\]
\[\theta^{\bar{\partial}^*}: \Lambda^{p}\bar{\mathbb{C}}^n\otimes \Lambda^{q-1}\mathbb{C}^n\hookrightarrow (\bar{\mathbb{C}}^n\oplus \mathbb{C}^n)\otimes\Lambda^{p}\bar{\mathbb{C}}^n\otimes\Lambda^{q}\mathbb{C}^n,
\]
\[\theta^{\bar{\partial}^*}(\omega\otimes\eta)=-\frac{1}{\sqrt{n-q+1}}\sum_{i=1}^{n}(\bar{e}_i,0)\otimes\omega\otimes e_i\wedge \eta.
\]
These intertwiners, just as in the Riemannian case, are isometric monomorphisms of $U(n)$ representations. One easily checks the following relations,
\[  \frac{-i}{\sqrt{p+1}}\sigma(\partial)\circ\theta^{\partial}=\id_{\Lambda^{p+1}\bar{\mathbb{C}}^n\otimes \Lambda^q\mathbb{C}^n},\quad\quad\quad \frac{-i}{\sqrt{n-p+1}}\sigma(\partial^*)\circ\theta^{\partial^*}=\id_{\Lambda^{p-1}\bar{\mathbb{C}}^n\otimes \Lambda^q\mathbb{C}^n},
\]
\[  \frac{-i}{\sqrt{q+1}}\sigma(\bar\partial)\circ\theta^{\bar\partial}=\id_{\Lambda^{p}\bar{\mathbb{C}}^n\otimes \Lambda^{q+1}\mathbb{C}^n},\quad\quad\quad \frac{-i}{\sqrt{n-q+1}}\sigma(\bar\partial^*)\circ\theta^{\bar\partial^*}=\id_{\Lambda^{p}\bar{\mathbb{C}}^n\otimes \Lambda^{q-1}\mathbb{C}^n}.
\]
The symbol maps $\sigma(\partial), \sigma(\partial^*), \sigma(\bar\partial), \sigma(\bar\partial^*)$ are the obvious  maps between vector spaces corresponding to the morphisms of vector bundles above.

Using Lemma \ref{linalg} (which works in the Hermitian case as well) one can check that  
\[ \Ker \sigma(\partial)=\Ima (\theta^{\partial})^\perp,\quad\quad\quad \Ker \sigma(\partial^*)=\Ima (\theta^{\partial^*})^\perp
\]
\[ \Ker \sigma(\bar\partial)=\Ima (\theta^{\bar\partial})^\perp,\quad\quad\quad \Ker \sigma(\bar\partial^*)=\Ima (\theta^{\bar\partial^*})^\perp.
\]
Combining this with the fact that the images of $\partial$ and $\partial^*$ are orthogonal and an analogous statement about $\bar{\partial}$ and $\bar{\partial}^*$, one gets the  claim about $L_1$ and $L_2$.

The constants of ellipticity for $L_1^*L_1$ and $L_2^*L_2$ are easy to compute from the relations
\[\partial^*\partial+\partial\partial^*=\bar{\partial}^*\bar{\partial}+\bar{\partial}\bar{\partial}^*=\frac{1}{2}\Delta.\]
Hence   for $p\notin\{0,n\}$ the constant for $L_1^*L_1$ is $\frac{1}{2}\min{\{\frac{1}{p+1},\frac{1}{n-p+1}\}}$ and similarly for $q\notin\{0,n\}$ the constant for $L_2^*L_2$ is $\frac{1}{2}\min{\{\frac{1}{q+1},\frac{1}{n-q+1}\}}$. 

When $p\in\{0,n\}$ then $L_1^*L_1=\frac{1}{2}\Delta$ and the constant is $\frac{1}{2}$. Similarly for $q\in{\{0,n\}}$, $L_2^*L_2=\frac{1}{2}\Delta$. One now chooses the smaller Kato constant from the ones provided by the inequalities induced by $L_1$ and by $L_2$.
\end{proof} 
\begin{remark} The Kato constants in the K\"{a}hler case provided by Theorem \ref{Kahler}  for a harmonic field $\omega$ of bidegree $(p,q)$ are smaller or equal than the Kato constants provided by Theorem \ref{main} except in the case $p=q$.
\end{remark}

A special case of the theorem is the following
\begin{corollary} \label{holforms} If $\omega\in\Gamma(\Lambda^{p,0}M)$ is a holomorphic $p$-form on a K\"{a}hler
 manifold then it satisfies the Kato inequality
 \[|d|\omega||\leq \frac{1}{\sqrt{2}}|\nabla \omega|.\]
 \end{corollary}
 \begin{proof} The operator $L_2=\bar{\partial}$ is injectively elliptic Stein-Weiss on $\Gamma(\Lambda^{p,0}M)$ and the constant of ellipticity  for $L_2^*L_2$ is $1/2$.
 \end{proof}
 One might ask whether this is the best we can do in the K\"{a}hler case with this technique. Unfortunately, the answer is yes.
\begin{lemma} There is no linear combination $a\partial+b\partial^*+c\bar{\partial}+d\bar{\partial}^*$ which is a Stein-Weiss operator and such that $(b\neq 0$ and $c\neq 0)$ or $(a\neq 0$ and $d\neq 0$.   
\end{lemma}
\begin{proof} The first thing to note is that the images of $\theta^{\partial^*}$ and $\theta^{\bar{\partial}}$ are not orthogonal. On the other hand by the complex version of Lemma \ref{linalg2} the maps $\theta^{\partial^*}$ and $\theta^{\bar{\partial}}$ are uniquely determined by the symbols of the operators $\partial^*$ and $\bar{\partial}$ up to a constant of modulus $1$. 
\end{proof}

Notice that the inequalities in Theorem \ref{Kahler} hold without any global condition on $\omega$ or $M$. We would like now to reprove the result from  \cite{KoLiZh,La} about $1$-forms mentioned in the introduction using Theorem \ref{Kahler}.   We have almost everything except for the fact that a harmonic $1$-field splits into a harmonic $(1,0)$-part and $(0,1)$-part.  The following proposition,  which is well known in the case of compact K\"{a}hler manifolds takes care of that (compare with Proposition $1$ in \cite{Ya}).
\begin{prop}\label{Yau}  Let $M$ be a complete K\"{a}hler manifold and let  $\omega\in\Gamma(\Lambda^kT^*M)$  be $L^2$ integrable. The following equations are equivalent.
\begin{enumerate} 
\item[(1)]  $(d+d^*)\omega=0$;
\item[(2)]  $(\partial+\partial^*)\omega=0$;
\item[(3)]  $ (\bar{\partial}+\bar{\partial}^*)\omega=0$.
\end{enumerate}
\end{prop}
\begin{proof} The identities
\[ \Delta=2\Delta_{\partial}=2\Delta_{\bar{\partial}}
\]
and "Bochner technique" solve the problem in the compact case. In the non-compact case it is enough to prove that  $\Delta_{\partial}\omega=0$, $\omega\in L^2$ implies that  $(\partial+\partial^*)\omega$ is $L^2$ integrable. The other cases are completely analogous. 

Indeed, since $\partial+\partial^*$ is formally adjoint we have that the domain of the maximal extension of $\partial+\partial^*$ contains the domain of the functional analytic adjoint of (the maximal extension of) $\partial+\partial^*$, on which they coincide. Hence 
\[ \int_{M}\langle (\partial+\partial^*)\eta,\omega\rangle=\int_{M}\langle \eta, (\partial+\partial^*)\omega \rangle
\]
for all $\eta,\omega\in L^2$ such that $(\partial+\partial^*)\eta, (\partial+\partial^*)\omega \in L^2$. Taking $\eta=(\partial+\partial^*)\omega$ solves the problem.

We are now using a special collection of cut-off functions $\psi_\nu$ with the following properties: there exists    a collection of compact subsets $K_{\nu}\subset K_{\nu +1}\subset M$ such that
\begin{equation}\label{cutoff} \psi_\nu\equiv 1 \mbox{ and } |d\psi_{\nu}|\leq 1 \mbox{ on } K_{\nu} \quad \mbox{ and }\quad \supp{\psi_{\nu}}\subset K_{\nu+1}.
\end{equation} Such a collection exists on a complete manifold by Hopf-Rinow lemma (see Proposition 8.1 in \cite{De}). We have
\begin{equation} \label{Boch}0= \int_{M}\langle \psi_{\nu}^2\omega,\Delta_{\partial}\omega\rangle=\int_{M}\langle\partial(\psi_{\nu}^2\omega),\partial\omega \rangle+\int_{M}\psi_\nu^2|\partial^*\omega|^2=\end{equation}\[=2\int_{M}\langle\partial\psi_{\nu}\wedge\omega,\psi_{\nu}\partial\omega\rangle+\int_{M}\psi_\nu^2|\partial\omega|^2+\int_{M}\psi_\nu^2|\partial^*\omega|^2.
\]
It follows from here that
\[\int_{M}\psi_\nu^2|\partial\omega|^2\leq 2\left|\int_{M}\langle\partial\psi_{\nu}\wedge\omega,\psi_{\nu}\partial\omega\rangle\right|\leq 2\left(\int_{M}|\partial \psi_{\nu}\wedge\omega|^2\right)^{1/2} \left(\int_{M}|\psi_{\nu}|^2|\partial\omega|^2\right)^{1/2}\]
Combining this with (\ref{cutoff}) we get \[ \int_{M}\psi_\nu^2|\partial\omega|^2\leq 4\int_{M}|\partial \psi_{\nu}\wedge\omega|^2\leq C\int_{M}|\partial \psi_{\nu}|^2|\omega|^2 \leq C\int_{M} |\omega|^2.
\]
Therefore $\partial\omega$ is $L^2$ integrable and by (\ref{Boch}), $\partial^*\omega$ is also $L^2$ integrable.
\end{proof}
The following  simple lemma is an interesting fact on its own.
\begin{lemma} \label{sum} Let $\phi_1\in\Gamma(E_1)$ and $\phi_2\in\Gamma(E_2)$  be sections of two vector bundles, $E_1$ and $E_2$, that satisfy the Kato inequalities
\[ |d|\phi_1|| \leq\alpha_1 |\nabla^{E_1} \phi_1|\quad\quad \mbox{ and }\quad\quad |d|\phi_2|| \leq\alpha_2 |\nabla^{E_2} \phi_2|,
\]
for some constants $\alpha_1$ and $\alpha_2\leq 1$. Then $(\phi_1,\phi_2)\in\Gamma(E_1\oplus E_2)$ satisfies the inequality
\[|d|(\phi_1,\phi_2)||\leq\max{\{\alpha_1,\alpha_2\}}|\nabla^{E_1\oplus E_2}(\phi_1,\phi_2)|.
\]
\end{lemma}
\begin{proof} \[    2|(\phi_1,\phi_2)|\cdot |d|(\phi_1,\phi_2)||=|d\left (|(\phi_1,\phi_2)|^2\right)|=|d|\phi_1|^2+d|\phi_2|^2|\leq 2|\phi_1|\cdot |d|\phi_1||+2|\phi_2|\cdot |d|\phi_2||\leq\] 
\[\leq 2\max{\{\alpha_1,\alpha_2\}}(|\phi_1|\cdot |\nabla^{E_1}\phi_1|+|\phi_2|\cdot|\nabla^{E_2}\phi_2|)\leq 2\max{\{\alpha_1,\alpha_2\}}|(\phi_1,\phi_2)|\cdot|\nabla^{E_1\oplus E_2}(\phi_1,\phi_2)|.\]
\end{proof}
Combining Theorem \ref{Kahler}, Proposition \ref{Yau} and Lemma \ref{sum} we get the following result that  appears in  \cite{KoLiZh} and \cite{La}.
\begin{corollary}\label{refineK} Let $\omega$ be a harmonic form of degree $1$ or $2n-1$ which is $L^2$ integrable on a complete K\"{a}hler manifold $M$ of complex dimension $n$. Then
\[ |d|\omega||\leq \frac{1}{\sqrt{2}}|\nabla \omega|.
\]
\end{corollary}
The following result characterizes the  equality case in Theorem \ref{Kahler} in the simplest of the situations.
\begin{prop} Let $\omega\in\Gamma(\Lambda^{0,q})$ and $\eta\in\Gamma(\Lambda^{p,0})$ be  harmonic fields on a K\"{a}hler manifold of dimension $n$. Then a necessary condition for equality for $\omega$ and $\eta$ in Theorem \ref{Kahler} is the existence of  real one-forms $\xi,\gamma\in\Gamma(T^*M)$  such that
\[\nabla\omega=\xi^{0,1}\otimes \omega \quad\quad\quad\quad \nabla\eta=\gamma^{1,0}\otimes \eta.
\]
If $\omega\in\Gamma(\Lambda^{n,q})$ and $\eta\in\Gamma(\Lambda^{p,n})$ then a necessary condition for equality is
\[ \nabla\omega=\xi^{1,0}\otimes  \omega\quad\quad\quad\quad \nabla\eta=\gamma^{0,1}\otimes \eta.
\]
\end{prop}
\begin{proof} The equality in Lemma \ref{elliptKato} happens when there exists a real valued function $f:M\rightarrow\mathbb{R}$ such that
\[ \nabla\omega=f\Pi^{\perp}(\xi_0\otimes\omega)= f\xi_0\otimes\omega-\Pi(f\xi_0\otimes \omega).
\]
Let $\xi:=f\xi_0$. This is a real $1$-form. In the case $\omega\in \Gamma(\Lambda^{0,q})$, the operator $\Pi$ is nothing else but $\xi\otimes \omega\mapsto \xi^{1,0}\wedge\omega$ where the element $\xi^{1,0}\wedge\omega$ is a section of $T^*_{\mathbb{C}}M\otimes\Lambda^{0,q}$ via the map
\[ \theta^{\partial}(\xi^{1,0}\wedge\omega)=\xi^{1,0}\otimes \omega.
\]
We have therefore the first claim and the second is entirely analogous. The third claim follows from the observation that
\[\sum_{i=1}^{n}(0,e_i)\otimes \bar{e}_i\wedge \iota_{(\xi^{0,1})^*}\omega=\xi^{0,1}\otimes\omega
\]
and the last one is similar to this.
\end{proof}


\begin{thebibliography}{99}
\bibliographystyle{siam}

\bibitem{BGV} 
N. Berline, E. Getzler, M. Vergne, {\it Heat Kernels and Dirac operators}
 Grundlehren der mathematischen Wissenschaften, Springer (1992), Berlin.
  \bibitem{Br}
T. Branson, {\it Kato constants in Riemannian geometry}, Math. Res.
Lett. {\bf 7} (2000), 245--261.
\bibitem{Br2}
T. Branson, {\it Stein-Weiss operators and ellipticity}, J. Funct. Anal. {\bf 151} (1997), 334-383.
 \bibitem{CaGaHe}
D. M. J. Calderbank, P. Gauduchon and M. Herzlich, {\it Refined Kato
inequalies and conformal weights  in Riemannian geometry}, J. Funct.
Anal. {\bf 173} (2000), 214--255.
\bibitem{De}
J.-P. Demailly, {\it $L^2$ estimates for the $\bar{\partial}$ operator on complex manifolds},
Notes de cours, Ecole d'ŽtŽ de MathŽmatiques (Analyse Complexe), Institut Fourier, Grenoble, Juin 1996, available at \url{http://www-fourier.ujf-grenoble.fr/~demailly/books.html}
 \bibitem{KoLiZh}
S. L. Kong, P. Li and D. T. Zhou, {\it Spectrum of Laplacian on
quaternionic K\"{a}hler manifolds}, J. Diff. Geom. {\bf 78} (2008),
295--332.
 \bibitem{La}
K. H. Lam, {\it Results on a weighted Poincar\'{e} inequality of
complete manifolds}, Trans. Amer. Math. Soc. {\bf 362} (2010),
5043--5062.
\bibitem{Wa}
X. Wang, {\it On the $L^2$-cohomology of a convex cocompact
hyperbolic manifold,} Duke Math. J. {\bf 115(2)} (2002), 311--327.
\bibitem{Wa2}
X. Wang, {\it On $L^2$ cohomology of ACH Kahler manifolds, 
} Proc. AMS 135 (2007), 2949-2960.
\bibitem{Ya}
S.T. Yau, {\it Some function-theoretic properties of complete Riemannian manifolds and their applications to geometry}, Indiana Univ. Math. J. {\bf 25(7)}, (1976), 659--670.
\end{thebibliography}
\end{document}